\documentclass[12pt]{amsart}

\usepackage{amsfonts, amsthm, amsmath}

\usepackage{rotating}
\usepackage{tikz}

\usepackage{amscd}

\usepackage[latin2]{inputenc}

\usepackage{t1enc}

\usepackage[mathscr]{eucal}

\usepackage{indentfirst}

\usepackage{graphicx}

\usepackage{graphics}

\usepackage{pict2e}

\usepackage{mathrsfs}

\usepackage{enumerate}
\usepackage[pagebackref]{hyperref}
\hypersetup{backref, pagebackref, colorlinks=true}
\usepackage{cite}
\usepackage{color}
\usepackage{epic}
\usepackage{hyperref} 
\numberwithin{equation}{section}
\theoremstyle{plain}

\newtheorem{theorem}{Theorem}
\newtheorem{definition}[theorem]{Definition}
\newtheorem{lemma}[theorem]{Lemma}

\newtheorem{cor}[theorem]{Corollary}
\newtheorem{remark}[theorem]{Remark}
\newtheorem{example}[theorem]{Example}

\def\inv{\mathrm{inv}}

\def\DES{\mathrm{DES}}
\def\EXC{\mathrm{EXC}}

\def\max{\mathrm{max}}
\def\min{\mathrm{min}}

\def\exc{\mathrm{exc}}
\def\dd{\mathrm{dd}}
\def\des{\mathrm{des}}

\def\S{\mathfrak{S}}
\def\L{\mathfrak{L}}

\def\area{\mathrm{area}}

\def\ppp{\operatorname{(31--2)}}
\def\pp{\operatorname{(2--13)}}

\def\val{\operatorname{{\bf val}}}
\def\pos{\operatorname{{\bf pos}}}
\def\area{\operatorname{{area}}}
\def\Nes{\operatorname{{\bf nes}}}

\DeclareMathOperator\cros{cros}
\DeclareMathOperator\nest{nest}

\def\boxit#1{\leavevmode\hbox{\vrule\vtop{\vbox{\kern.33333pt\hrule
    \kern1pt\hbox{\kern1pt\vbox{#1}\kern1pt}}\kern1pt\hrule}\vrule}}

\usepackage{collectbox}



\usepackage{xcolor}

\allowdisplaybreaks

\begin{document}

\title[A new encoding of permutations]
{A new encoding of permutations by \\Laguerre histories}

\author[S.H.F. Yan]{Sherry H.F. Yan}
\address[Sherry H.F.  Yan]{Department of Mathematics,
Zhejiang Normal University, Jinhua 321004, P.R. China}
\email{huifangyan@hotmail.com}

\author[H. Zhou]{Hao Zhou}
\address[Hao Zhou]{Department of Mathematics,
Zhejiang Normal University, Jinhua 321004, P.R. China}

\author[Z. Lin]{Zhicong Lin}
\address[Zhicong Lin]{School of Science, Jimei University, Xiamen 361021, P.R. China}
\email{zhicong.lin@univie.ac.at}

\date{\today}

\begin{abstract}
We construct  a bijection from permutations to some weighted Motzkin paths known as Laguerre histories. As one application of our bijection,  a neat  $q$-$\gamma$-positivity expansion of the $(\inv,\exc)$-$q$-Eulerian polynomials is obtained. 
\end{abstract}



\maketitle


  \section{Introduction}
 A {\em Motzkin path} of length $n$ is a lattice path in the first  quadrant starting from $(0,0)$, ending at $(n,0)$, with three possible steps: $U=(1,1)$ (up step), $L=(1,0)$ (level step) and $D=(1,-1)$ (down step). A {\em$2$-Motzkin path} is a Motzkin path in which each level step is labelled by $L_0$ or $L_1$. The $2$-Motzkin paths will be represented as words over the alphabet $\{U,D,L_0,L_1\}$. A {\em Laguerre history} of length $n$ is a pair $(w,\mu)$ such that $w=w_1\cdots w_n$ is a $2$-Motzkin path and $\mu=(\mu_1,\cdots,\mu_n)$ is a vector satisfying $0\leq\mu_i\leq h_i(w)$, where 
 $$h_i(w):=\#\{j\mid j<i, w_j=U\}-\# \{j\mid j<i, w_j=D\}$$
  is the {\em height} of the $i$-th step of $w$. Denote by $\L_n$ the set of all Laguerre histories of length $n$. It is known that the cardinality of $\L_n$ is $(n+1)!$.
  
Laguerre histories can be used to encode permutations. Two famous such encodings  in the literature are known as the Fran\c{c}on--Viennot bijection and the Foata--Zeilberger bijection; see~\cite{csz} for the relationship between these two bijections and~\cite{csz,Shin-zeng1,Shin-zeng2012} for other modifications of them. The purpose of this paper is to present a new encoding of permutations by Laguerre histories with an interesting application to the $q$-$\gamma$-positivity expansion of the $(\inv,\exc)$-$q$-Eulerian polynomials. The inspiration of our bijection comes from the recent works by Cheng--Elizalde--Kasraoui--Sagan~\cite{ceks}, Lin~\cite{Lin1} and Lin--Fu~\cite{lin-fu}.
We need some further  definitions and notations before we can state our main results. 

Let $\S_n$ be the set of permutations of $[n]:=\{1,2,\ldots, n\}$. For any permutation $\sigma\in \S_n$, written as the word $\sigma=\sigma(1)\sigma(2)\cdots \sigma(n)$, the entry $i\in [n]$ is called an {\em excedance} (resp.~\emph{descent}, {\em double descent}) of $\sigma$ if $i<\sigma(i)$ (resp. $\sigma(i)>\sigma(i+1)$, $\sigma(i-1)>\sigma(i)>\sigma(i+1)$). Here we use the convention $\sigma(0)=\sigma(n+1)=0$ when considering double descents of $\sigma$. Denote by $\exc(\sigma)$, $\des(\sigma)$ and $\dd(\sigma)$ the numbers of excedances, descents and double descents of $\sigma$, respectively. 
It is well known that the  {\em Eulerian polynomials} $A_n(t)$ has the interpretations (cf.~\cite[Sec.~1.3]{st0}):
 \begin{equation}\label{eqe}
 A_n(t)=\sum_{\sigma\in \S_n}q^{\des(\sigma)}=\sum_{\sigma\in \S_n}q^{\exc(\sigma)}.
\end{equation}
Foata and  Sch\"uzenberger~\cite[Theorem~5.6]{fsc0} proved the following elegant $\gamma$-positivity expansion of the Eulerian polynomials
\begin{equation}\label{gam:eul}
A_n(t)=\sum_{k=0}^{\lfloor \frac{n-1}{2}\rfloor}|DD_{n,k}|t^k(1+t)^{n-1-2k},
\end{equation}
where 
$DD_{n,k}:=\{\sigma\in \S_n: \des(\sigma)=k, \dd(\sigma)=0\}$. Recently, different refinements of~\eqref{gam:eul} and other $\gamma$-positive polynomials arising in enumerative and geometric combinatorics have been widely studied; the reader is referred to the survey of Athanasiadis~\cite{Ath} and the book exposition by Petersen~\cite{Petersen2015} for more information.

For $\sigma\in \S_n$, let $\inv(\sigma):=|\{(i,j)\in[n]\times[n]: i<j, \sigma(i)>\sigma(j)\}|$ be the {\em inversion number} of $\sigma$.  The statistic $(31$--$2)\sigma$ (resp.~$(2$--$13)\sigma$) is the number of pairs $(i,j)$ such that $2\leq i<j\leq n$ and $\sigma(i-1)>\sigma(j)>\sigma(i)$ (resp.~$\sigma(j)>\sigma(i)>\sigma(j-1)$). Shin and Zeng~\cite[Theorem~1]{Shin-zeng2016} proved the following $q$-analog of~\eqref{gam:eul} for the $(\inv,\exc)$-$q$-Eulerian polynomials. 

\begin{theorem}[Shin and Zeng]\label{coro4}
For $n\geq 1$, we have 
$$
\sum_{\sigma\in \S_n}q^{\inv(\sigma)-\exc(\sigma)}t^{\exc(\sigma)}=\sum_{k=0}^{\lfloor{n-1\over 2}\rfloor}\biggl(\sum_{\sigma\in DD_{n,k}}q^{2\,(2-13)\sigma+(31-2)\sigma}\biggr)t^{k}(1+t)^{n-1-2k}.
$$
\end{theorem}

For a $2$-Motzkin path $w=w_1\cdots w_n$ of length $n$, define 
\begin{align*}
U(w):=\{i\in[n-1]: w_i=U\}\quad\text{and}\quad L_1(w):=\{i\in[n]: w_i=L_1\}.
\end{align*}
Let $\area(w)$ be the {\em area} between $w$ and the $x$-axis. Our new encoding of permutations by Laguerre histories is a generalization of the bijection in~\cite[Lemma~16]{lin-fu} between $321$-avoiding permutations and $2$-Motzkin paths. 

\begin{theorem}\label{thm:main}
There is a bijection $\Phi: \S_n\rightarrow\L_{n-1}$ such that if $\Phi(\sigma)=(w,\mu)$, then 
\begin{equation}\label{exc:inv}
\EXC(\sigma)=U(w)\uplus L_1(w)\quad\text{and}\quad\inv(\sigma)-\exc(\sigma)=\area(w)+\sum_{i=1}^{n-1}\mu_i,
\end{equation}
where $\EXC(\sigma)$ is the set of excedances of $\sigma$.
\end{theorem}

An index $i\in[n-1]$ is called a {\em shifted double excedance} of $\sigma$ if $i<\sigma(i)$ and $\sigma^{-1}(i+1)<i+1$. Let $DE_{n,k}$ be the set of permutations $\sigma\in\S_n$ with
\begin{itemize}
\item no shifted double excedances,
\item $\exc(\sigma)=k$.
\end{itemize}
For example, for $n=4$, we have 
\begin{equation}\label{DE4}
DE_{4,0}=\{1234\}\quad\text{and}\quad DE_{4,1}=\{1423,1432,3124,3214,4123,4132,4213,4231\}.
\end{equation}
As one application of our encoding $\Phi$, the following neat $q$-$\gamma$-positivity expansion, different with that in Theorem~\ref{coro4}, for the $(\inv,\exc)$-$q$-Eulerian polynomials is derived.
 
\begin{theorem}\label{gam:main}
For $n\geq 1$, we have
$$
\sum_{\sigma\in \S_n}q^{\inv(\sigma)-\exc(\sigma)}t^{\exc(\sigma)}=\sum_{k=0}^{\lfloor{n-1\over 2}\rfloor}\biggl(\sum_{\sigma\in DE_{n,k}}q^{\inv(\sigma)-\exc(\sigma)}\biggr)t^{k}(1+t)^{n-1-2k}.
$$
\end{theorem}
As an example of Theorem~\ref{gam:main}, for $n=4$, it follows from~\eqref{DE4} that 
\begin{equation*}
\sum_{\sigma\in \S_4}q^{\inv(\sigma)-\exc(\sigma)}t^{\exc(\sigma)}=(1+t)^3+(2q+3q^2+2q^3+q^4)t(1+t).
\end{equation*} 

We will also provide an alternative approach to Theorem~\ref{coro4} by combining our bijection $\Phi$ and a modified version of the Fran\c{c}on--Viennot bijection. Denote by $\mathfrak{A}_n$  the set of permutations $\sigma\in\S_n$ that is down-up:
$$
\sigma(1)>\sigma(2)<\sigma(3)>\sigma(4)<\cdots.
$$
It is well known (cf.~\cite{Stanley}) that $|\mathfrak{A}_{2k-1}|$ is the {\em$k$-th tangent number} $T_k$, which appears in the Taylor expansion
$$
\tan(t)=\sum_{k\geq1}T_k\frac{t^{2k-1}}{(2k-1)!}=t+ 2\frac{t^3}{3!}+16\frac{t^5}{5!}+272\frac{t^7}{7!}+7936\frac{t^9}{9!}+\cdots.
$$
Setting $t=-1$ in Theorem~\ref{coro4} we recover the following result about {\em$q$-tangent numbers} due to Shin and Zeng~\cite[Theorem~3]{Shin-zeng1}.

\begin{cor}[Shin and Zeng]
For $n\geq1$, we have 
$$
\sum_{\sigma\in \S_n}(-1/q)^{\exc(\sigma)}q^{\inv(\sigma)}=
\begin{cases}
0\quad&\text{if $n$ is even,}\\
(-1)^{\frac{n-1}{2}}\sum\limits_{\sigma\in\mathfrak{A}_n}q^{2\,(2-13)\sigma+(31-2)\sigma}\quad&\text{if $n$ is odd.}
\end{cases}
$$
\end{cor}

In the same vein, setting $t=1$ in Theorem~\ref{gam:main} gives the following new interpretation of  the above $q$-tangent numbers. 
\begin{cor}
For $n\geq1$, we have 
$$
\sum_{\sigma\in \S_n}(-1/q)^{\exc(\sigma)}q^{\inv(\sigma)}=
\begin{cases}
0\quad&\text{if $n$ is even,}\\
(-1/q)^{\frac{n-1}{2}}\sum\limits_{\sigma\in DE_{n,(n-1)/2}}q^{\inv(\sigma)}\quad&\text{if $n$ is odd.}
\end{cases}
$$
\end{cor}
\begin{remark}
The set $DE_{2k+1,k}$ is a new combinatorial model for the tangent numbers. For instance,  
$DE_{5,2}$ consists of 16 permutations: 
\begin{align*}
&54231,54213,54123,54132,54312,54321,45231,45213,\\
&45123,45132,45312,45321,32541,32514,31524,31542.
\end{align*}
Although combining our bijection $\Phi$ and a modification of the Fran\c{c}on--Viennot bijection $\Psi$ (introduced in Section~\ref{sec:3}) will set up a link between $DE_{2k+1,k}$ and $\mathfrak{A}_{2k+1}$, no direct bijection between these two models is known. 
\end{remark}

 The rest of this paper is organized as follows. In Section~\ref{sec:2}, we construct the bijection $\Phi$ and prove Theorem~\ref{thm:main}. In Section~\ref{sec:3}, we introduce a simple group action on Laguerre histories and prove Theorems~\ref{coro4} and~\ref{gam:main}.

\section{The construction of $\Phi$}
\label{sec:2}

 In this section, we will construct the bijection $\Phi$ and prove Theorem~\ref{thm:main}. The following definition is important in constructing the bijection $\Phi$. 
 
\begin{definition}
For $k\in [n]$ and $\sigma\in\S_n$, the crossing index and nesting index on $k$ of  $\sigma$ are defined, respectively, by
\begin{align*}
  \cros_k(\sigma)&:=\#\{\ell\mid \sigma(k)<\ell\leq k<\sigma^{-1}(\ell)\,\,  \mbox{or}\, \, \ell<k< \sigma(\ell)<\sigma(k)\},\\
 \nest_k(\sigma)&:=\#\{\ell\mid \ell<\sigma(k)\leq k<\sigma^{-1}(\ell)\,\,  \mbox{or}\, \, \ell<k<\sigma(k)<\sigma(\ell)\}.
 \end{align*}
 Denote by  $\cros(\sigma)=\sum_{k=1}^{n}\cros_k(\sigma)$ (resp.~$\nest(\sigma)=\sum_{k=1}^{n}\nest_k(\sigma)$) the crossing (resp.~nesting) number of $\sigma$.
  \end{definition}
  
 We also need two vectors to keep track of the values and positions of  excedances. For a permutation $\sigma\in\S_n$, let
  $$
  \val(\sigma):=(v_1, v_2, \ldots, v_n)\quad\mbox{and}\quad\pos(\sigma):=(p_1, p_2, \ldots, p_n)
  $$
  where $v_i=\chi(i>\sigma^{-1}(i))$ and $p_i=\chi(\sigma(i)>i)$. Let 
  $$
  \Nes(\sigma):=(\nest_1(\sigma),\nest_2(\sigma),\ldots,\nest_{n}(\sigma))
  $$
be the vector that keeps  track of the nesting indices of $\sigma$. Define the mapping $\Phi(\sigma)=(w,\mu)$, where for $i\in[n-1]$,  $\mu_i=\nest_i(\sigma)$ and 
 $$
  w_i=\left\{
  \begin{array}{ll}
  U& \mbox{if} \,\, v_{i+1}=0 \,\,\mbox{and }\,\, p_{i}=1, \\
 D& \mbox{if} \,\, v_{i+1}=1 \,\,\mbox{and }\,\, p_{i}=0, \\
   L_1& \mbox{if} \,\, v_{i+1}=p_i=1, \\
  L_0& \mbox{if} \,\, v_{i+1}=p_i=0.
  \end{array}
  \right.
  $$
  \begin{example}\label{exam:8}
  Take $\sigma={\bf43}21{\bf89}765\in\S_9$, then $\val(\sigma)=(0,0,1,1,0,0,0,1,1)$, $\pos(\sigma)=(1,1,0,0,1,1,0,0,0)$ and $\Nes(\sigma)=(0,1,1,0,0,0,2,1,0)$. Thus, $\Phi(\sigma)=(w,\mu)\in\L_{8}$ with $w=UL_1DL_0UUDD$ and $\mu=(0,1,1,0,0,0,2,1)$. 
\end{example}

  We are going to prove that $\Phi$ is a bijection between $\S_n$ and $\L_{n-1}$ satisfying~\eqref{exc:inv}. The following lemma plays an essential role in proving $\Phi$ is a bijection.
  
\begin{lemma}\label{mainlem1}
Suppose that  $v=(v_1, v_2,\ldots, v_n)$ and  $p=(p_1, p_2, \ldots, p_n)$ are  two $0$-$1$ vectors with the same number of zeros , and  $\mu=(\mu_1, \mu_2, \ldots, \mu_{n})$ is a vector of nonnegative integers.  Then  $v=\val(\sigma)$, $p=\pos(\sigma)$ and $\mu=\Nes(\sigma)$ for a unique permutation $\sigma\in\S_n$ if and only if for each $k\in[n]$,
\begin{equation}\label{ineq}
0\leq \mu_k\leq\#\{\ell : v_{\ell}=0, \ell\leq k\} -\#\{\ell : p_{\ell}=0, \ell<k\}-1.
\end{equation}
\end{lemma}
\begin{proof}
For convenience, we set 
$$
d_k(v,p):=\#\{\ell : v_{\ell}=0, \ell\leq k\} -\#\{\ell : p_{\ell}=0, \ell<k\}
$$
and 
$$
\tilde{d}_k(v,p):=\#\{\ell : v_{\ell}=1, \ell>k\} -\#\{\ell : p_{\ell}=1, \ell>k\}.
$$
Since the $0$-$1$ vectors $v$ and $p$ has the same number of zeros, we have 
$$
\tilde{d}_k(v,p)=\#\{\ell : v_{\ell}=0, \ell\leq k\}-\#\{\ell : p_{\ell}=0, \ell\leq k\}.
$$
It then follows that 
\begin{equation}\label{eq:pk1}
\tilde{d}_k(v,p)=d_k(v,p)\quad\text{whenever $p_k=1$}. 
\end{equation}

First we prove the ``only if'' side. We distinguish two cases:
\begin{itemize}
\item If $p_k=0$, that is $\sigma(k)\leq k$, then  $d_k(v,p)$ counts the number of indices  $\ell$ such that $\ell\leq k$ and $\sigma^{-1}(\ell)\geq k$. Thus,  $d_k(v,p)\geq \nest_k(\sigma)+1$ according to the definition of $\nest_k(\sigma)$.
\item Otherwise, we have $p_k=1$ and in this case $\tilde{d}_k(v,p)$ counts the number of indices $\ell$ such that $\ell\leq k$ and $\sigma(\ell)>k$. It follows from the definition of $\nest_k(\sigma)$ that $\tilde{d}_k(v,p)\geq \nest_k(\sigma)+1$.
\end{itemize}
In view of~\eqref{eq:pk1}, we get  $d_k(v,p)\geq \nest_k(\sigma)+1$ in either case. 

It remains to prove the ``if'' side of the lemma. Given two  $0$-$1$ vectors $v,p$ and a vector $\mu$ satisfying~\eqref{ineq}, we will construct the unique permutation $\sigma\in\S_n$ such that $v=\val(\sigma)$, $p=\pos(\sigma)$ and $\mu=\Nes(\sigma)$.  The value of $\sigma(k)$ for $p_k=0$ are determined by the following steps:
\begin{itemize}
\item[(1)] Set $P=\{i\in[n]:p_i=0\}$ and $V=\{i\in[n]: v_i=0\}$;
\item[(2)] let $k\leftarrow \min(P)$ (this means that $\min(P)$ is assigned to $k$) and find index $j$ such that $j$ is the $(\mu_k+1)$-th smallest value in $V$;
\item[(3)] set $\sigma(k)=j$, $P\leftarrow P\setminus\{k\}$ and $V\leftarrow V\setminus\{j\}$; go to step (2) if $P$ is not empty. 
\end{itemize} 
In step (2) of the above algorithm, since $\mu_k+1\leq d_k(v,p)$, $j$ must exist and $j\leq k$. In order to have $\nest_k(\sigma)=\mu_k$  the value of $\sigma(k)$  must be $j$. Take $v=(0,0,1,1,0,0,0,1,1)$, $p=(1,1,0,0,1,1,0,0,0)$ and $\mu=(0,1,1,0,0,0,2,1,0)$ for example, the above algorithm determines $\sigma(3)=2$, $\sigma(4)=1$, $\sigma(7)=7$, $\sigma(8)=6$ and $\sigma(9)=5$, successively. Similarly, the value of $\sigma(k)$ for $p_k=1$ are determined by the following steps:
\begin{itemize}
\item[(a)] Set $P=\{i\in[n]:p_i=1\}$ and $V=\{i\in[n]: v_i=1\}$;
\item[(b)] let $k\leftarrow \max(P)$ and find index $j$ such that $j$ is the $(\mu_k+1)$-th greatest  value in $V$;
\item[(c)] set $\sigma(k)=j$, $P\leftarrow P\setminus\{k\}$ and $V\leftarrow V\setminus\{j\}$; go to step (b) if $P$ is not empty. 
\end{itemize} 
In step (b) of the above algorithm, since $\mu_k+1\leq d_k(v,p)=\tilde{d}_k(v,p)$ (in view of~\eqref{eq:pk1}), $j$ must exist and $j>k$. In order to have $\nest_k(\sigma)=\mu_k$  the value of $\sigma(k)$ must be $j$. Continuing with the above running example, we determine $\sigma(6)=9$, $\sigma(5)=8$, $\sigma(2)=3$ and $\sigma(1)=4$, successively. Finally, the permutation $\sigma$ constructed by the above two algorithms is $432189765$, which coincides with the one in Example~\ref{exam:8}. The proof of the ``if'' side is complete.
\end{proof}

Now, we are ready to prove Theorem~\ref{thm:main}.
 \begin{proof}[{\bf Proof of Theorem~\ref{thm:main}}] It follows from the construction of $\Phi$ that Laguerre histories of length $n-1$ are in bijection with the triples $(v,p,\mu)$, where $v=(v_1, v_2,\ldots, v_n)$ and  $p=(p_1, p_2, \ldots, p_n)$ are  two $0$-$1$ vectors with the same number of zeros , and  $\mu=(\mu_1, \mu_2, \ldots, \mu_{n})$ is a vector satisfying~\eqref{ineq}.  The latter objects are in bijection with $\S_n$ by Lemma~\ref{mainlem1} and so $\Phi$ is a bijection between $\S_n$ and $\L_{n-1}$. 
 
 Next we show that $\Phi$ has the required properties~\eqref{exc:inv}. The first equality of~\eqref{exc:inv} is clear from the definition of $\Phi$. For the second equality of~\eqref{exc:inv}, we claim that 
 \begin{equation}\label{claim}
 h_i(w)=\cros_i(\sigma)+\nest_i(\sigma).
 \end{equation}
 Invoking the relationship (see \cite{Shin-zeng1}, Eq.~(40)) 
 $$
 \inv(\sigma)=\exc(\sigma)+\cros(\sigma)+2\nest(\sigma)
 $$
 and the simple facts 
 $$
\area(w)=\sum_{i=1}^{n-1}h_{i}(w), \qquad\nest(\sigma)=\sum_{i=1}^n\mu_i
$$
we see immediately that claim~\eqref{claim} implies 
$$\inv(\sigma)-\exc(\sigma)=\area(w)+\sum_{i=1}^n\mu_i,$$
as desired. 

It remains to show the claim~\eqref{claim}. 
 We proceed the proof by considering the following two cases.
 \begin{itemize}
 \item Case 1: $w_i=D$ or $L_0$. In this case, we have $p_{i}=0$ and 
  \begin{align*}
  \cros_i(\sigma)+\nest_i(\sigma)&=\#\{\ell:  \ell\leq  i<\sigma^{-1}(\ell)\} \\
  &= \#\{\ell: \ell\leq i, v_{\ell}=0\}-\#\{\ell\mid \ell\leq i, p_{\ell}=0\}\\
  &= \#\{\ell: 2\leq \ell\leq i, v_{\ell}=0\}-\#\{\ell: \ell<i, p_{\ell}=0\}\\
  &= \#\{\ell: \ell<i, w_{\ell}=U\}-\#\{\ell: \ell<i, w_{\ell}=D\}\\
  &=h_i(w).
  \end{align*}

  \item Case 2: $w_i=U$ or $L_1$. In this case, we have $p_i=1$ and
 \begin{align*}
  \cros_i(\sigma)+\nest_i(\sigma)&=\#\{\ell: \ell< i< \sigma(\ell)\} \\
  &= \#\{\ell: \ell<i, p_{\ell}=1\}-\#\{\ell: \ell\leq i, v_{\ell}=1\}\\
  &= \#\{\ell: \ell< i, w_{\ell}=U\}-\#\{\ell: \ell<i, w_{\ell}=D\}\\
  &=h_i(w).
 \end{align*}
   \end{itemize}
 Hence we have deduced~\eqref{claim}, which ends the proof of Theorem~\ref{thm:main}.
 \end{proof} 
   
\section{Applications of $\Phi$}%
\label{sec:3}

 Lin~\cite{Lin1} introduced a group action on $2$-Motzkin paths in the sprit of the Foata--Strehl action on permutations. Here we generalize it to Laguerre histories. Let $i\in [n]$ and $(w,\mu)\in \L_n$. If the $i$-th step of $w$ is level, then let $\varphi_i(w,\mu)=(w', \mu)$, where $w'$ is the   2-Motzkin path  obtained from $w$ by changing the label of the $i$-th step. Otherwise, define $\varphi_i(w,\mu)=(w, \mu)$. For any subset $S\subseteq [n]$ define the function $\phi_{S}: \L_n\rightarrow \L_n$ by $\varphi_S(w,\mu)=\prod _{i\in S} \varphi_i(w,\mu)$. Hence the group $\mathbb{Z}_2^n$ acts on $\L_n$ via the function $\varphi_S$.

 This action divides the set $\L_n$ into disjoint orbits and each orbit has a unique Laguerre history which has all its level steps labelled by $L_0$. Let us introduce 
 $$
 \mathcal{O}_{n,k}=\{(w,\mu)\in \L_n: \text{ all level steps of $w$ are $L_0$},  \#U(w)=k\}.
 $$
 It then follows from this action that 
 \begin{equation}\label{action}
 \sum_{(w,\mu)\in\L_n}t^{\#U(w)+\#L_1(w)}q^{\area(w)+\sum_i\mu_i}=\sum_{k=0}^{\lfloor\frac{n}{2}\rfloor}\biggl(\sum_{(w,\mu)\in \mathcal{O}_{n,k}}q^{\area(w)+\sum_i\mu_i}\biggr)t^k(1+t)^{n-2k}.
 \end{equation}
 
 We are now in position to prove Theorem~\ref{gam:main}.

 \begin{proof}[{\bf Proof of Theorem~\ref{gam:main}}] It is clear from the construction of $\Phi$ that a permutation $\sigma\in\S_n$ has no shifted double excedances if and only if $w$ has no $L_1$ step, where $(w,\mu)=\Phi(\sigma)$. Theorem~\ref{gam:main} then follows from Theorem~\ref{thm:main} and expansion~\eqref{action}.
 \end{proof}
 
 For the proof of Theorem~\ref{coro4}, we still need a modified version of the Francon--Viennot bijection (cf.~\cite{Shin-zeng2012}) $\Psi:\S_n\rightarrow\L_{n-1}$ that we recall next.  For $\sigma\in\S_n$, define 
the refinements of two generalized patterns by
\begin{align*}
\ppp_k\sigma&=\#\{i: i+1<j\text{ and } \sigma(i+1)<\sigma(j)=k<\sigma(i)\},\\
\pp_k\sigma&=\#\{i: i-1>j\text{ and } \sigma(i-1)<\sigma(j)=k<\sigma(i)\}.
\end{align*}
Let us use the assumption $\sigma(0)=\sigma(n+1)=0$ and introduce $\Psi(\sigma)=(w,\mu)\in\L_{n-1}$, where for each $i\in[n-1]$: 

  $$
  w_i=\left\{
  \begin{array}{ll}
  U& \mbox{if} \,\, \sigma(j-1)>\sigma(j)=i<\sigma(j+1),  \\
 D&\mbox{if} \,\, \sigma(j-1)<\sigma(j)=i>\sigma(j+1),  \\
   L_1&\mbox{if} \,\, \sigma(j-1)>\sigma(j)=i>\sigma(j+1),
    \\
  L_0 &\mbox{if} \,\,\sigma(j-1)<\sigma(j)=i<\sigma(j+1),
  \end{array}
  \right.
  $$
  and
$\mu_i=\pp_i(\sigma)$.
The bijection $\Psi$ has the following features
\begin{equation}\label{des}
 \DES(\sigma)=U(w)\uplus L_1(w)\quad\text{and}\quad 2\pp\sigma+\ppp\sigma=\area(w)+\sum_{i=1}^{n-1}\mu_i,
 \end{equation}
 where $\DES(\sigma)$ is the set of descents of $\sigma$.
 
 \begin{proof}[{\bf Proof of Theorem~\ref{coro4}}] It is clear from the construction of $\Psi$ that a permutation $\sigma\in\S_n$ has no double descents if and only if $w$ has no $L_1$ step, where $(w,\mu)=\Psi(\sigma)$.   Applying $\Psi$ to both sides of~\eqref{action} and using~\eqref{des} yields 
 \begin{equation}\label{gam:des}
\sum_{\sigma\in \S_n}t^{\des(\sigma)}q^{2\,(2-13)\sigma+(31-2)\sigma}=\sum_{k=0}^{\lfloor{n-1\over 2}\rfloor}\biggl(\sum_{\sigma\in DD_{n,k}}q^{2\,(2-13)\sigma+(31-2)\sigma}\biggr)t^{k}(1+t)^{n-1-2k}.
\end{equation}
 On the other hand, Theorem~\ref{thm:main} together with the properties~\eqref{des} of $\Psi$ gives
 \begin{equation}\label{exc:des}
 \sum_{\sigma\in \S_n}t^{\exc(\sigma)}q^{\inv(\sigma)-\exc(\sigma)}=\sum_{\sigma\in \S_n}t^{\des(\sigma)}q^{2\,(2-13)\sigma+(31-2)\sigma}.
 \end{equation}
 Theorem~\ref{coro4} then follows by combining~\eqref{gam:des} and~\eqref{exc:des}.
 \end{proof}

\section*{Acknowledgments}
 This work was supported by the National Science Foundation of China grants 11671366, 11871247 and 11501244, and the Training Program Foundation for Distinguished Young Research Talents of Fujian Higher Education.


\begin{thebibliography}{99}

\bibitem{Ath} C.A. Athanasiadis, Gamma-positivity in combinatorics and geometry, S\'em. Lothar. Combin. 77 (2018), Article B77i, 64pp (electronic).
 
 \bibitem{ceks} S.-E. Cheng, S. Elizalde, A. Kasraoui and B.E. Sagan, Inversion polynomials for $321$-avoiding permutations, Discrete Math., {\bf313} (2013), 2552--2565.
 
 \bibitem{csz} R.J. Clarke, E. Steingr\'imsson, and J. Zeng, New Euler-Mahonian statistics on permutations and words, Adv. in Appl. Math., \textbf{18} (1997), 237--270.
 
 \bibitem{fsc0}D. Foata and M.-P. Sch\"uzenberger, {\em Th\'eorie G\'eom\'etrique des Polyn\^omes Eul\'eriens},  Lecture Notes in Mathematics, Vol. 138, Springer-Verlag, Berlin, 1970.


\bibitem{Lin1}
Z. Lin,   On $\gamma$-positive polynomials arising in pattern avoidance, Adv. in Appl. Math., {\bf 82} (2017), 1--22.

\bibitem{lin-fu} Z. Lin and S. Fu, On $1212$-avoiding restricted growth functions, Electron. J. Combin., {\bf 24(1)} (2017), \#P1.53.

\bibitem{Petersen2015}
T.K. Petersen,  {\em Eulerian Numbers}. With a foreword by Richard Stanley.  Birkh\"auser Advanced Texts: Basler Lehrb\"ucher. Birkh\"auser/Springer, New York, 2015.

\bibitem{Shin-zeng1}
H. Shin, J. Zeng,  The $q$-tangent and $q$-secant numbers via continued fractions, European J. Combin.,   {\bf 31} (2010), 1689--1705.

\bibitem{Shin-zeng2012}
H. Shin, J. Zeng, The symmetric and unimodal expansion of Eulerian polynomials via continued fractions, European J. Combin., {\bf 33} (2012), 111--127.

\bibitem{Shin-zeng2016}
H. Shin, J. Zeng,   Symmetric unimodal expansion of excedances in colored permutations, European J. Combin.,   {\bf 52} (2016), 174--196.

\bibitem{st0} R.P. Stanley,  {\em Enumerative Combinatorics Vol. 1}, Cambridge Studies in Advanced Mathematics 49, Cambridge University Press, Cambridge,  1997.

\bibitem{Stanley}
R.P. Stanley,  A Survey of Alternating Permutations, in Combinatorics
and Graphs, R.A. Brualdi et. al. (eds.), Contemp. Math., Vol. 531,
 Amer. Math. Soc., Providence, RI, 2010, pp.~165--196.


\end{thebibliography}
 \end{document}